\documentclass[12pt]{amsart}
\usepackage{Baskervaldx}
\usepackage[baskervaldx,cmintegrals,bigdelims,vvarbb]{newtxmath}
\usepackage[cal=cm]{mathalfa}
\usepackage{enumerate, geometry, pgfplots, graphicx, subcaption, amsmath, array, pdfsync, tikz, tikz-cd, color, url, float, wrapfig, comment, MnSymbol}
\usepackage{comment}

\usetikzlibrary{positioning}

\usepackage{colortbl}
\usepackage{tabulary}
\usepackage{etoolbox}
\usepackage{tabularx}
\usepackage{xfrac}
\newcolumntype{Y}{>{\centering\arraybackslash}X}
\usepackage{makecell}

\newcommand{\tabsymbol}[1]{%
  \multicolumn{1}{c@{\makebox[0pt]{#1}}}{}%
}

\usepackage[hyperfootnotes=false, pdfencoding=auto]{hyperref}
\hypersetup{
  colorlinks= true,
  linkcolor=[RGB]{0,89,42},
  urlcolor=[RGB]{0,89,42},
  citecolor=[RGB]{0,89,42}
}

\usepackage[font=footnotesize]{caption}

\usetikzlibrary{calc}

\numberwithin{equation}{section}

\title{Fano hypersurfaces with no finite order birational automorphisms}
\author{Nathan Chen, Lena Ji, and David Stapleton}

\definecolor{mycolor}{RGB}{146, 214, 203}
\definecolor{myothercolor}{RGB}{179, 215, 232}

\newtheorem{theorem}{Theorem}

\newtheorem{proposition}[theorem]{Proposition}
\numberwithin{theorem}{section}

\newtheorem{Lthm}{Theorem}

\newtheorem{Lcor}[Lthm]{Corollary}

\newcommand{\hr}[2]{\hyperref[#1]{#2}}

\theoremstyle{definition}
\newtheorem{remark}[theorem]{Remark}

\newtheorem{construction}[theorem]{Construction}
\newtheorem{question}[theorem]{Question}
\newtheorem*{question*}{Question}

\newtheorem{example}[theorem]{Example}
\newtheorem{definition}[theorem]{Definition}

\def\ZZ{{\mathbb Z}}

\def\CC{{\mathbb C}}

\def\PP{{\mathbb{P}}}

\def\Pic{{\mathrm{Pic}}}

\def\Spec{{\mathrm{Spec} \ }}

\def\Bir{{\mathrm{Bir}}}

\def\Cr{{\mathrm{Cr}}}

\def\ra{{\rightarrow}}
\def\etabar{{\overline{\eta}}}

\def\Frac{{\mathrm{Frac} \ }}

\def\PGL{{\mathrm{PGL}}}

\def\Aut{{\mathrm{Aut}}}

\def\sp{{\mathrm{sp}}}

\DeclareMathOperator{\HH}{H}

\newcommand*{\sheafTor}{\mathcal{T}\kern -.5pt or}
\newcommand*{\Extc}{\mathcal{E}\kern -.5pt xt}
\newcommand{\BirSpecialization}{\Xi_\eta(Z)}

\pgfplotsset{compat=1.15}
\definecolor{graycolor}{rgb}{0.66,0.66,0.66}

\newcommand{\diam}{\diamond}

\thanks{During the preparation of this article, the first author was partially supported by an NSF postdoctoral fellowship DMS-2103099, the second author was partially supported by NSF grant DMS-1840234, and the third author was partially supported by NSF grant DMS-1952399.}


\begin{document}
\maketitle

\thispagestyle{empty}

\begin{abstract}
We use the specialization homomorphism for the birational automorphism group to study finite order birational automorphisms. For a family of varieties over a DVR, we prove that a birational automorphism of order coprime to the residue characteristic cannot specialize to the identity. As an application, we show that very general $n$-dimensional hypersurfaces of degree $d \geq 5 \lceil(n+3)/6 \rceil$ have no finite order birational automorphisms.
\end{abstract}

The birational automorphism group of a variety $X$---denoted $\Bir(X)$---is one of the most natural birational invariants associated to $X$. For $X = \PP^n_{\CC}$, the \textit{Cremona group} $\Cr_n(\CC) = \Bir(\PP^n_{\CC})$ is an object of classical and modern interest, and it is extremely interesting and complicated when $n\ge 2$. Beyond the case of projective space, it is natural to study the birational automorphism group of a smooth degree $d$ hypersurface $X\subset \PP^{n+1}_{\CC}$. In the general type case, if $n \geq 2$ and $d\ge n+3$, then $K_{X}$ is ample and Matsumura \cite{Matsumura63} showed that $\Bir(X)$ is equal to the automorphism group $\Aut(X)$. If $d=n+2$, in which case $X$ is Calabi--Yau (with Picard rank \(1\) if \(n \geq 3\)), then again \(\Bir(X)=\Aut(X)\) \cite{MatsusakaMumford64} (see also \cite[Lem A.1]{LS22}). However, if $d\le n+1$, in which case $X$ is Fano, very little is known about birational automorphisms in general once $n\ge 4$.

The most striking known result is the case of degree $d=n+1$ Fano hypersurfaces. To briefly summarize, there has been a great deal of work by many authors---including Fano, Segre, Iskovskikh, Manin, Pukhlikov, Corti, Cheltsov, de Fernex, Ein, Musta\c{t}\u{a}, and Zhuang---to show that if $n\ge 3$ and $d=n+1$, then any such smooth $X$ is birationally superrigid. As a consequence of their work, $\Bir(X) = \Aut(X)$ (see \cite{Kollar19} for a survey of the main ideas that were developed over time). In the case $d=n$, Pukhlikov used similar techniques to show that such hypersurfaces also satisfy $\Bir(X) = \Aut(X)$ once $n\ge 14$ \cite[Cor. 1]{Pukhlikov-index2}. For a smooth hypersurface $X$, having $\Bir(X) = \Aut(X)$ places strong constraints on the groups, as shown by Matsumura and Monsky \cite[Thm. 2 and Thm. 5]{MM63}: (1) if $n \geq 2$ and $d \geq 3$ (excluding the case $(n, d) = (2, 4)$), then $\Aut(X)$ is naturally identified with a finite subgroup of $\Aut(\PP^{n+1}_{\CC}) = \PGL_{n+2}(\CC)$, and (2) if $n \geq 2$, $d \geq 3$, and $X$ is very general, then $\Aut(X)$ is trivial. There seem to be few known restrictions on \(\Bir\) when $d<n$.

We first prove a result about specializing finite order elements in the birational automorphism group (Proposition~\ref{prop:specialization-order-l}). For a family of varieties over a complex curve, this shows that a nontrivial finite order birational automorphism cannot specialize to the identity on the central fiber. We apply our result to hypersurfaces, but we believe that this specialization method will also be useful for studying the birational automorphism groups of other varieties.

By degenerating to a reducible hypersurface, our result will imply that if a very general non-ruled degree \(d\) hypersurface has no \(p\)-torsion in its birational automorphism group, then the same also holds in degree \(d+1\). By degenerating to positive characteristic---following the work of Koll\'{a}r \cite{Koll'ar-hypersurfaces}---we can control the torsion in $\Bir(X)$ for certain hypersurfaces in the Fano range.

\begin{Lthm}\label{thm:torsion-p^e}
Let \(p\) be a prime and let \(n\) and \(d\) be integers; if $p=2$ further assume that $n$ is even. Let \(X\subset\mathbb P^{n+1}_{\mathbb C}\) be a very general hypersurface. If $d \geq p\left\lceil\frac{n+3}{p+1}\right\rceil$, then any finite order element in \(\Bir(X)\) has order \(p^r\) for some \(r\).
\end{Lthm}
\noindent When $d\ge n+2$, Theorem~\ref{thm:torsion-p^e} is well known as $K_X$ is ample or trivial. When $d=n+1$, or when \(d=n\) and \(n\geq 14\), we use the known results on index one and two Fano hypersurfaces. So our contribution to Theorem~\ref{thm:torsion-p^e} is for Fano hypersurfaces of degree $d\le n$.

We achieve the largest range of degrees in which we can apply Theorem~\ref{thm:torsion-p^e} by choosing the smallest primes. This gives the following corollary.

\begin{Lcor}\label{cor:no-torsion}
Let \(X\subset\mathbb P^{n+1}_{\mathbb C}\) be a very general degree \(d\) hypersurface. If either
\begin{enumerate} \item \(d\geq 3\lceil\frac{n+3}{4}\rceil\) and \(n\) is even, or
\item \(d\geq 5\lceil\frac{n+3}{6}\rceil\) and \(n\) is odd,
\end{enumerate}
Then \(\Bir (X)\) has no elements of finite order.
\end{Lcor}

The table below places our results in the context of previous work for some values of $(n, d)$. The number \(2\) means that any finite order element in \(\Bir(X)\) has order a power of \(2\) (possibly order \(1\)) as a result of Theorem~\ref{thm:torsion-p^e}, and similarly for $3$.

\begin{table}[ht]
\tiny
\renewcommand{\arraystretch}{0.7}
\begin{tabularx}{15cm}{c|YYYYYYYYYYYYYYYYYYYYYY}
\tabsymbol{\small \emph{d}} & & & & & & & & & & & & & & & & & & & & & & \\

\textbf{21} & $\cdot$ & $\cdot$ & $\cdot$ & $\cdot$ & $\cdot$ & $\cdot$ & $\cdot$ & $\cdot$ & $\cdot$ & $\cdot$ & $\cdot$ & $\cdot$ & $\cdot$ & $\cdot$ & $\cdot$ & $\cdot$ & $\cdot$      & $\diam$ & $\ast$ & $\blacksquare$ & 3 & $\blacksquare$ \\

\textbf{20} & $\cdot$ & $\cdot$ & $\cdot$ & $\cdot$ & $\cdot$ & $\cdot$ & $\cdot$ & $\cdot$ & $\cdot$ & $\cdot$ & $\cdot$ & $\cdot$ & $\cdot$ & $\cdot$ & $\cdot$ & $\cdot$      & $\diam$ & $\ast$  & $\blacksquare$      & 2 &   & 2 \\

\textbf{19} & $\cdot$ & $\cdot$ & $\cdot$ & $\cdot$ & $\cdot$ & $\cdot$ & $\cdot$ & $\cdot$ & $\cdot$ & $\cdot$ & $\cdot$ & $\cdot$ & $\cdot$ & $\cdot$ & $\cdot$      & $\diam$ & $\ast$  & $\blacksquare$       & 3      & 2 &   & 2 \\

\textbf{18} & $\cdot$ & $\cdot$ & $\cdot$ & $\cdot$ & $\cdot$ & $\cdot$ & $\cdot$ & $\cdot$ & $\cdot$ & $\cdot$ & $\cdot$ & $\cdot$ & $\cdot$ & $\cdot$      & $\diam$ & $\ast$  & 3       & $\blacksquare$       & 3      & 2 &   & 2 \\

\textbf{17} & $\cdot$ & $\cdot$ & $\cdot$ & $\cdot$ & $\cdot$ & $\cdot$ & $\cdot$ & $\cdot$ & $\cdot$ & $\cdot$ & $\cdot$ & $\cdot$ & $\cdot$      & $\diam$ & $\ast$  & 2       &         & 2       &        &   &   & \\

\textbf{16} & $\cdot$ & $\cdot$ & $\cdot$ & $\cdot$ & $\cdot$ & $\cdot$ & $\cdot$ & $\cdot$ & $\cdot$ & $\cdot$ & $\cdot$ & $\cdot$      & $\diam$ & $\ast$  & 3       & 2       &         & 2       &        &   &   & \\

\textbf{15} & $\cdot$ & $\cdot$ & $\cdot$ & $\cdot$ & $\cdot$ & $\cdot$ & $\cdot$ & $\cdot$ & $\cdot$ & $\cdot$ & $\cdot$      & $\diam$ & $\ast$  & $\blacksquare$       & 3       & 2       &         &         &        &   &   & \\

\textbf{14} & $\cdot$ & $\cdot$ & $\cdot$ & $\cdot$ & $\cdot$ & $\cdot$ & $\cdot$ & $\cdot$ & $\cdot$ & $\cdot$      & $\diam$ & $\ast$  &         & 2       &         & 2       &         &         &        &   &   & \\

\textbf{13} & $\cdot$ & $\cdot$ & $\cdot$ & $\cdot$ & $\cdot$ & $\cdot$ & $\cdot$ & $\cdot$ & $\cdot$      & $\diam$ & 3       & 2       &         &         &         &         &         &         &        &   &   & \\

\textbf{12} & $\cdot$ & $\cdot$ & $\cdot$ & $\cdot$ & $\cdot$ & $\cdot$ & $\cdot$ & $\cdot$      & $\diam$ & $\blacksquare$       & 3       & 2       &         &         &         &         &         &         &        &   &   & \\

\textbf{11} & $\cdot$ & $\cdot$ & $\cdot$ & $\cdot$ & $\cdot$ & $\cdot$ & $\cdot$      & $\diam$ &         & 2       &         &         &         &         &         &         &         &         &        &   &   & \\

\textbf{10} & $\cdot$ & $\cdot$ & $\cdot$ & $\cdot$ & $\cdot$ & $\cdot$      & $\diam$ & 2       &         & 2       &         &         &         &         &         &         &         &         &        &   &   & \\

\textbf{9}  & $\cdot$ & $\cdot$ & $\cdot$ & $\cdot$ & $\cdot$      & $\diam$ & 3       &         &         &         &         &         &         &         &         &         &         &         &        &   &   & \\

\textbf{8}  & $\cdot$ & $\cdot$ & $\cdot$ & $\cdot$      & $\diam$ & 2       &         &         &         &         &         &         &         &         &         &         &         &         &        &   &   & \\

\textbf{7}  & $\cdot$ & $\cdot$ & $\cdot$      & $\diam$ &         &         &         &         &         &         &         &         &         &         &         &         &         &         &        &   &   & \\

\textbf{6}  & $\cdot$ & $\cdot$      & $\diam$ & 2       &         &         &         &         &         &         &         &         &         &         &         &         &         &         &        &   &   & \\

\textbf{5}  & $\cdot$      & $\diam$ &         &         &         &         &         &         &         &         &         &         &         &         &         &         &         &         &        &   &   & \\

\textbf{4}  & $\diam$ &         &         &         &         &         &         &         &         &         &         &         &         &         &         &         &         &         &        &   &   & \smash{\raisebox{-4pt}{\rlap{\qquad \small \emph{n}}}} \\
\hline
\rule{0pt}{\normalbaselineskip} & \textbf{3} & \textbf{4} & \textbf{5} & \textbf{6} & \textbf{7} & \textbf{8} & \textbf{9} & \textbf{10} & \textbf{11} & \textbf{12} & \textbf{13} & \textbf{14} & \textbf{15} & \textbf{16} & \textbf{17} & \textbf{18} & \textbf{19} & \textbf{20} & \textbf{21} & \textbf{22} & \textbf{23} & \textbf{24}
\end{tabularx}
\caption*{$\blacksquare$ = Corollary~\ref{cor:no-torsion} \quad $\cdot$ = \text{non-Fano} \quad $\diam$ = \text{Fano index } 1 \quad $\ast$ = \text{Pukhlikov's Fano index 2 results}}
\end{table}

\noindent Restricting the possible orders of torsion elements places strong restrictions on the birational automorphism group. Since the Cremona group contains $p$-torsion for any $p$, Theorem~\ref{thm:torsion-p^e} with \(p=2\) if \(n\) is even and \(p=3\) if \(n\) is odd implies that \(\Bir(X)\not\cong\Cr_n(\CC)\) if \(d\geq 2\lceil\frac{n+3}{3}\rceil\) for \(n\) even and \(d\geq 3\lceil\frac{n+3}{4}\rceil\) for \(n\) odd. Remarkably, Cantat proved $\Bir(X)\not\cong\Cr_n(\CC)$ whenever $X$ is \emph{any} irrational variety \cite[Thm.~C]{Cantat14}.

\begin{remark}
The parity assumption on $n$ in Theorem~\ref{thm:torsion-p^e} comes from studying the singularities of odd dimensional double covers of hypersurfaces in characteristic \(2\) \cite[Thm.~C]{ChenStapleton-rational-endomorphisms}. At the moment, we cannot give an explicit resolution in this case.
\end{remark}

In light of Corollary~\ref{cor:no-torsion}, which shows that \(\Bir(X)\) contains no finite order elements, one might wonder how far apart this is from showing that \(\Bir(X)=\{1\}\). (Recall that \(\Aut(X)=\{1\}\) for these hypersurfaces.) There are a number of related works in this vein.
Any finite order element of \(\Bir(X)\) is regularizable, i.e. it is equivalent to a regular automorphism on a birational model of \(X\). For surfaces and for birationally rigid Fano threefolds, the regularizable automorphisms generate the birational automorphism group (e.g. \(\Cr_2(\CC)\) is generated by \(\Aut(\PP^2)=\PGL_3(\CC)\) and the Cremona involution). Cheltsov has asked whether this holds in general \cite[Conj.~1.12]{Cheltsov04}. Recently, Lin and Shinder \cite{LS22} proved that this is false by showing that for \(n\geq 3\), \(\Cr_n(\CC)\) is not generated by (pseudo-)regularizable elements.

Throughout the paper we consider $\Bir$ as a group, not as a group scheme. However, for non-uniruled varieties Hanamura has several results on giving $\Bir$ a scheme structure \cite{Hanamura87, Hanamura88}.

\noindent\textbf{Notation.} $R$ will denote a DVR with field of fractions $K=\Frac R$ and residue field $k$. We will write \(\eta\) for the generic point  of $\Spec R$ and \(0\) for the closed point.

\noindent\textbf{Outline.}
Let $X$ be a family over $R$ and let $Z \subset X_{0}$ be a component of the special fiber. In $\S 1$, we first identify a subgroup $\Xi_{\eta}(Z) \subset \Bir_K(X_{\eta})$, consisting of the birational automorphisms of the generic fiber $X_{\eta}$ that "specialize". We construct a specialization homomorphism
\[ \sp_\eta \colon  \Xi_{\eta}(Z) \rightarrow \Bir_{k}(Z). \]
Next, we study torsion in the birational automorphism group in $\S 2$ and show that if $\ell$ is a positive integer that is invertible in $R$, then the kernel of the specialization map cannot contain birational automorphisms $\phi \in \Xi_{\eta}(Z)$ of order $\ell$ (see Proposition~\ref{prop:specialization-order-l}\eqref{item:mu_l-nontrivial-on-special-fiber}). In $\S 3$ we degenerate to characteristic $p > 0$ and take advantage of some nice properties that are satisfied by the special fiber $X_{0}$ to show that $\Xi_{\eta}(Z)$ coincides with $\Bir(X_{\eta})$. In particular, we use the fact (building on work of Koll\'{a}r \cite{Koll'ar-hypersurfaces} and of the first and third authors \cite{ChenStapleton-finite-BirX}) that certain $p$-cyclic covers in characteristic $p$ have no birational automorphisms. This is finally applied to families of hypersurfaces to prove Theorem~\ref{thm:torsion-p^e} and Corollary~\ref{cor:no-torsion}.

\noindent\textbf{Acknowledgements.}
We are grateful to J\'er\'emy Blanc, Michel Brion, Serge Cantat, J\'anos Koll\'ar, Davesh Maulik, Aleksandr Pukhlikov, Evgeny Shinder, Burt Totaro, Ziquan Zhuang, and Susanna Zimmermann for helpful conversations. The first author would like to thank Professor Pietro Pirola and the University of Pavia for the opportunity to visit and their warm hospitality, during which parts of this paper were drafted.

\section{The specialization homomorphism for \texorpdfstring{$\Bir$}{\texttwoinferior}}\label{sec:Specialization}

The specialization homomorphism was first defined by Matsusaka and Mumford (who attribute it to Artin) \cite{MatsusakaMumford64}, and it has also appeared in the literature for surfaces \cite[\S 3.1]{Persson77} \cite[\S 2]{LieblichMaulik18}. To our knowledge, it has not previously been applied to systematically study birational automorphisms.

\begin{definition}[{\cite[Thm I]{MatsusakaMumford64}}]\label{defn:specializes}
Let \(X_R\) be an integral flat separated scheme over \(R\), and let \(Z\subset X_0\) be a reduced irreducible component that appears with multiplicity one in the central fiber. Let \(\phi\in\Bir_K(X_\eta)\) be a birational automorphism, and let \(\Gamma\subset X_R\times_R X_R\) be the closure of the graph of \(\phi\).
We say \(\phi\) \emph{specializes to \(Z\)} if the special fiber \(\Gamma_0\) has a unique component that maps birationally to \(Z\) under both projections.
\end{definition}

\begin{example}
In the ruled setting, a birational automorphism of \(X_\eta\) need not specialize.
For the automorphism \(x\mapsto\tfrac{t}{x}\) on the generic fiber of the constant family \(\PP^1_x\times\mathbb A^1_t\to\mathbb A^1_t\), the special fiber \(\Gamma_0\) has two irreducible components, each of which is contracted under one of the projections.
\end{example}

\begin{definition}[{\cite[Def. 1.1, Def. 1.5]{ChenStapleton-rational-endomorphisms}}]\label{defn:sustained-modifications}
A normal scheme \(X\) has \emph{(separably uni-)ruled modifications} if every exceptional divisor of every normal birational modification \(Y\to X\) is (separably uni-)ruled.
A normal scheme $X_R$ has \textit{sustained (separably uni-) ruled modifications} if there exists a generically finite extension of DVRs $R\subset R'$ such that for every generically finite extension of DVRs $R'\subset S$, the normalization of $X_{S}$ has (separably uni-)ruled modifications. Here we fix an algebraic closure of \(K\), and the ring extension \(R\subset R'\) being generically finite means that \(\Frac R'\) is a finite algebraic extension of \(K\).
\end{definition}

\begin{proposition}[The specialization homomorphism]\label{prop:specialization-hom} Let $X_R$ be an integral flat separated scheme over \(R\) and \(Z\subset X_0\) a reduced irreducible component appearing with coefficient one in the special fiber.
\begin{enumerate}
\item\label{item:find-open}
If $\phi$ is a birational automorphism of $X_\eta$ that specializes to $Z$, then there are open sets $U_1, U_2\subset X_R$ such that each \(U_i\) meets $Z$, $\phi$ gives an isomorphism between $U_1$ and $U_2$, and the restriction of $\phi$ to $X_{0}$ is an isomorphism:
\[
\phi|_{X_0 \cap U_1} \colon  Z\cap U_1\cong Z\cap U_2.
\]
\item\label{item:specialization-homomorphism}
The set of birational automorphisms that specialize to \(Z\) forms a subgroup of \(\Bir_K(X_\eta)\), which we denote \(\BirSpecialization\). There is a specialization group homomorphism:
\[
\sp_\eta \colon  \BirSpecialization \ra \Bir_k(Z).
\]
\item\label{item:geometric-specialization}
Assume \(X_\eta\) and \(Z\) are geometrically integral over \(K\) and \(k\), respectively. The group \(\Xi_{\overline{\eta}}(Z_{\overline{k}})\) is the colimit of \(\Xi_{\eta'}(Z_{k'})\) over generically finite extensions \(R\subset R'\) of DVRs; thus, there is an induced specialization homomorphism:
\[ \sp_\etabar \colon  \Xi_{\overline{\eta}}(Z_{\overline{k}})\ra \Bir_{\overline{k}}(Z_{\overline{k}}). \]
\item\label{item:Bir-specializes-nonruled}
\cite[IV Ex.~1.17.3]{Koll'ar-rational-curves}
Assume that \(X_R\) is proper and has (separably uni-)ruled modifications, and that $Z$ is the unique irreducible component of \(X_0\) that is not (separably uni-)ruled. Then every birational automorphism of \(X_\eta\) specializes to \(Z\). That is, \(\BirSpecialization=\Bir_K(X_\eta)\).
\item\label{item:specialization-sustained}
In the setting of \eqref{item:Bir-specializes-nonruled}, assume furthermore that $X_R$ has sustained (separably uni-)ruled modifications; that \(X_\eta\) and \(Z\) are geometrically integral over $K$ and $k$, respectively; and that \(Z_{\overline{k}}\) is not (separably uni-)ruled over \(\overline{k}\). Let $R\subset R''$ be a generically finite extension of DVRs. Then \(\Bir(X_\eta)\) is a subgroup of \(\Bir(X_{\eta''})\), and there is a further generically finite extension \(R''\subset R'\) of DVRs such that the diagram commutes:
\[
\begin{tikzcd}
\Bir_K(X_\eta)\arrow[d]\arrow[r,"\sp_\eta"]& \Bir_k(Z)\arrow[d]\\
\Bir_{K'}(X_{\eta'})\arrow[r,"\sp_{\eta'}"]& \Bir_{k'}(Z_{k'}).
\end{tikzcd}
\]
In particular, there is a homomorphism
\[ \sp_\etabar \colon  \Bir_{\overline{K}}(X_\etabar)\ra \Bir_{\overline{k}}(Z_{\overline{k}}). \]
\end{enumerate}
\end{proposition}

\begin{proof}
Let \(\tilde{\phi}\colon X_R\dashrightarrow X_R\) be the birational map over \(R\) obtained from the closure \(\Gamma\subset X_R\times_R X_R\) of the graph of \(\phi\). Let \(\Gamma_0'\) be the unique component of \(\Gamma_0\) mapping birationally to \(Z\) under both projections, \(U_1\subset X_R\) the largest open subset on which \(\tilde{\phi}\) is an isomorphism, and \(U_2=\tilde{\phi}(U_1)\). Each \(U_i\) meets \(Z\) by maximality, and \(\phi_0=\tilde{\phi}|_{Z}\) as rational maps. This proves~\eqref{item:find-open}.

For~\eqref{item:specialization-homomorphism}, it is clear that the identity on \(X_\eta\) specializes to the identity on \(Z\).
If \(\phi\) specializes to \(Z\), then so does \(\phi^{-1}\) by exchanging the first and second projections. It remains to show that if \(\phi\) and \(\psi\) specialize to \(Z\), then so does \(\psi\circ\phi\), and that the specialization of the composition is the composition of the specializations.
For this, let \(\phi,\psi\in\BirSpecialization\), and let \({\tilde{\phi}}|_{U_{1,\tilde{\phi}}}\colon U_{1,\tilde{\phi}}\to U_{2,\tilde{\phi}}\) and \({\tilde{\psi}}|_{U_{1,\tilde{\psi}}}\colon U_{1,\tilde{\psi}}\to U_{2,\tilde{\psi}}\) be morphisms defined on the largest open subsets on which \({\tilde{\phi}}\) and \({\tilde{\psi}}\), respectively, induce isomorphisms. Let \(U_2=U_{2,\tilde{\phi}}\cap U_{1,\tilde{\psi}}, U_1={\tilde{\phi}}^{-1}(U_2),\) and \(U_3={\tilde{\psi}}(U_2)\). Then each \(U_i\cap Z\neq\emptyset\), so the assertion follows from the fact that \({\tilde{\psi}}|_{U_2\cap Z}\circ{\tilde{\phi}}|_{U_1\cap Z}=({\tilde{\psi}}\circ{\tilde{\phi}})|_{U_1\cap Z}\).

For \eqref{item:geometric-specialization}, if \(R\subset R'\) is a generically finite extension of DVRs, then \(X_{\eta'}\) and \(Z'\coloneqq Z\otimes_k k'\) are both integral and \(Z'\) has coefficient one in the special fiber, so they satisfy the assumptions in Definition~\ref{defn:specializes}. If \(\Gamma\) is the closure of the graph of an element of \(\BirSpecialization\), then by assumption it has a unique component mapping birationally to \(Z\) under the projections. Therefore, the base change to \(k'\) gives a component of the special fiber of \(\Gamma\otimes_R R'\) birational to \(Z_{k'}\) under the projections, and there is a unique such component since \(\Gamma\otimes_R R'\to R'\) is flat. This proves that \(\BirSpecialization\) is a subgroup of \(\Xi_{\eta'}(Z')\).

Before showing \eqref{item:Bir-specializes-nonruled}, first suppose that \(Y_R\) and \(Y'_R\) are flat integral schemes over \(\Spec (R)\) such that \(Y_R\) has (separably uni-)ruled modifications, every non-(separably uni-)ruled component of \(Y_0\) appears with coefficient one in \(Y_0\), \(Y'_R\) is proper, and \(Y'_0\) has a unique irreducible component \(Z'\) that is not (separably uni-)ruled. Then any birational map \(\phi\colon Y_\eta\dashrightarrow Y'_{\eta}\) induces a birational map \(\phi_0\colon Z\dashrightarrow Z'\) from some component \(Z=Z_{\phi}\) of \(Y_{0}\) that is not (separably uni-)ruled (c.f. \cite[IV Ex.~1.17]{Koll'ar-rational-curves}).
For this claim, first observe that the assumption on the coefficients of \(Y_0\) implies that the local ring at the generic point of every non-(separably uni-)ruled component of \(Y_0\) is a DVR. Now let \(\Gamma\) be the closure of the graph of \(\phi\) in \(Y_{R}\times_R Y'_{R}\), and let \(\Gamma_0'\) be the unique component of \(\Gamma_0\) mapping birationally to \(Z'\). Since \(Y_{R}\) has (separably uni-)ruled modifications and \(Z'\) is not (separably uni-)ruled, then \(\Gamma_0'\) maps birationally to a component \(Z\) of \(Y_{0}\), so the composition \(\phi_0 \colon Z\dashrightarrow\Gamma_0'\dashrightarrow Z'\) is a birational map.

We will now apply this to $X_R = Y_R =X'_{R}$ and $Z = Z'$ to prove \eqref{item:Bir-specializes-nonruled}. Let \(U_1\subset X\) be the largest open subset on which \(\tilde{\phi}\) is an isomorphism, and let \(U_2=\tilde{\phi}(U_1)\). Note that \(\tilde{\phi}^{-1}(U_2\cap X_{0})=U_1\cap X_{0}\), so each \(U_i\) meets \(Z\) by maximality, and \(\phi_0=\tilde{\phi}|_{Z}\) as rational maps.

For \eqref{item:specialization-sustained}, let \(R\subset R'\) be as in Definition~\ref{defn:sustained-modifications}. After replacing \(R'\) by a localization of its integral closure in \(K'\otimes_K K''\) we may assume \(R\subset R''\subset R'\). Then \(X_{\eta'}\) and \(Z'\coloneqq Z_{k'}\) are integral, and \(Z'\) appears with coefficient one in the central fiber of \(X_{R'}\), so the local ring of \(X_{R'}\) at the generic point of \(Z'\) is a DVR. Thus, the normalization \(X_{R'}^\nu\to X_{R'}\) is an isomorphism at the generic point of \(Z'\), so on the special fiber there is a component \(W\) mapping birationally to \(Z'\).
Now we apply \eqref{item:Bir-specializes-nonruled} to obtain a specialization map
\[\Bir_{K'}(X_{\eta'})=\Bir_{K'}(X_{\eta'}^\nu) \to \Bir_{k'}(W) = \Bir_{k'}(Z').\]
\eqref{item:specialization-sustained} then follows from \eqref{item:geometric-specialization}.
\end{proof}

Let \(X_R\) be a family of smooth proper varieties. The previous proposition describes how to specialize birational automorphisms, and one may wonder what the image of the subgroup \(\Aut_K(X_\eta)\cap\BirSpecialization \subset \Bir_K(X_\eta)\) is in \(\Bir_k(X_0)\).
Let \(\phi \in \Aut_K (X_\eta) \cap \BirSpecialization\).
If there is an ample divisor \(\mathcal L\) on \(X_\eta\) such that \(\mathcal L\) and \(\phi^*\mathcal L\) both extend to relatively ample divisors on the family \(X_R\), then a theorem of Matsusaka and Mumford shows that \(\phi\) extends to a (regular) automorphism \(\tilde{\phi}\in\Aut_R(X_R)\) and that \(\sp_\eta(\phi)\) is a (regular) automorphism of \(X_0\) \cite[Cor.~1]{MatsusakaMumford64}. Without this additional assumption that \(\phi\) preserves an ample class, one may ask:

\begin{question}
Is there a smooth proper family $X_R$ and an element $\phi \in \Aut(X_\eta)\cap \Xi_\eta(X_0)$ such that $\sp_\eta(\phi)\in \Bir(X_0)$ is not a regular automorphism?
\end{question}

\noindent In the next section, we will give an example of a family of K3 surfaces and an element \(\iota\in\Aut_K(X_\eta)\) which does not extend to a regular automorphism in \(\Aut_R(X_R)\) (Example~\ref{exmp:K3-example}). In our example \(\Aut(X_0)=\Bir(X_0)\), so \(\sp_\eta(\iota)\) is still a regular automorphism of \(X_0\).

\section{Kernel of the specialization homomorphism}
In this section, we study the kernel of the specialization homomorphism from \S\ref{sec:Specialization}. After regularizing an order \(\ell\) birational automorphism on a birational model of \(X_R\), our argument shows that any component of the special fiber fixed by the \(\ZZ/\ell\ZZ\) group action must be a multiple component.

\begin{proposition}\label{prop:specialization-order-l}
Let \(X_R\) be an integral flat separated scheme over \(R\) and \(Z\subset X_0\) an irreducible component.
Let \(\phi\in\BirSpecialization\) be a birational automorphism of order $\ell$, for some integer $\ell > 1$.
\begin{enumerate}
\item\label{item:automorphism-on-open}  There is an affine open $U\subset X_R$ meeting $Z$ on which \(\phi\) induces an automorphism over \(R\).
\item\label{item:quotient-by-mu_l} If $\ell$ is invertible in $R$, then the quotient $U/\langle\phi\rangle$ exists and \((U/\langle\phi\rangle)_0=(U\cap Z)/\langle\sp_\eta(\phi)\rangle\).
\item\label{item:mu_l-nontrivial-on-special-fiber} If $\ell$ is invertible in $R$, then \(\sp_\eta(\phi)\) has order \(\ell\) in \(\Bir_k(X_0)\). In particular, $\phi\not\in\ker(\sp_\eta)$.
\end{enumerate}
\end{proposition}

\begin{proof}
Let \(\tilde{\phi}\in\Bir_R(X_R)\) be induced by \(\phi\), and set \(U=\bigcap_{i=1}^{\ell-1}\tilde{\phi}^i(U')\), where \(U'\subset U_1\cap U_2\) is an affine open subset meeting \(Z\), and \(U_1\) and \(U_2\) are as in Proposition~\ref{prop:specialization-hom}\eqref{item:find-open}. This shows \eqref{item:automorphism-on-open}.

Now let \(U=\Spec A\), and let \(\phi\in\Aut_R(A)\) denote the induced automorphism. We write \((-)^\phi\) to mean the submodule of \(\phi\)-invariant elements of an \(A\)-module.
The quotient \(\Spec(A^\phi)\) is integral and normal \cite[Thm.~4.16]{Moonen-AV}, and it remains to show that \((A^\phi)\otimes_R k \cong (A\otimes_R k)^\phi\).

Left exactness of \((-)^\phi\) implies \(A^\phi/(\pi A)^{\phi}\hookrightarrow(A\otimes_R k)^\phi\) is injective. Since \(\phi\) is an automorphism over \(R\) and \(R\to A\) is flat, we have \((\pi A)^\phi=\pi(A^\phi)\) and \(A^\phi/(\pi A)^\phi\cong(A^\phi)\otimes_R k\), where \(\pi\) is a uniformizer of \(R\). This shows injectivity of \((A^\phi)\otimes_R k\to(A\otimes_R k)^\phi\). For surjectivity, let \(a\in (A\otimes_R k)^\phi\) and let \(\tilde{a}\in A\) be a lift. Then \(\frac{1}{\ell}\sum_{i=0}^{\ell-1}\phi^i(\tilde{a})\) is an element of \(A^\phi\) mapping to \(a\in A\otimes_R k\). This shows \eqref{item:quotient-by-mu_l}.

For \eqref{item:mu_l-nontrivial-on-special-fiber}, let \(G=\langle\phi\rangle\subset \Aut_R(U)\).
Since the quotient morphism \(q\colon U\to(U/G)\) is a morphism over \(R\), the pullback of the effective Cartier divisor \((U/G)_0\) on \(U/G\) is \(U_0\) \cite[\href{https://stacks.math.columbia.edu/tag/01WV}{Tag 01WV},\href{https://stacks.math.columbia.edu/tag/0C4U}{Tag 0C4U}]{stacks-project}. The projection formula \cite[Ch. 9 Prop. 2.11]{Liu-AG-book} yields
\(
q_*[q^*((U/G)_0)]=\ell[(U/G)_0].
\)
The restriction of \(q\) to \(U_0\) is thus a finite morphism of degree \(\ell\), so the order of \(\sp_\eta(\phi)\) in \(\Bir(X_0)\) must be \(\ell\).
\end{proof}

\begin{remark}
Proposition~\ref{prop:specialization-order-l}\eqref{item:mu_l-nontrivial-on-special-fiber} shows that the kernel of the specialization homomorphism does not contain any torsion of order coprime to the characteristic of the residue field of \(R\). In some special cases, the kernel is even trivial: when the specialization homomorphism \(\Pic(X_{\overline{\eta}})\to\Pic(X_{\overline{0}})\) is an isomorphism and \(\HH^0(X_{\overline{0}},\mathcal T_{X_{\overline{0}}})=0\), Lieblich and Maulik show using the Matsusaka--Mumford theorem \cite[Cor.~1]{MatsusakaMumford64} and a deformation theory argument that the specialization homomorphism is injective \cite[\S 2]{LieblichMaulik18}.
\end{remark}

However, injectivity does not hold in general. We now give a series of examples exhibiting nontrivial elements in the kernel of the specialization map.

\begin{example}\label{exmp:cremona-example}
Let \(k\) be a field, and let \(P_1,P_2,Q_t\in\mathbb P^2(k)\) be points such that \(P_1,P_2,Q_t\) are not collinear for \(t\neq 0\), but \(P_1,P_2,Q_0\) lie on a common line \(L\). Denote the subscheme \(P_1 + P_2 + Q_t\) by \(Y_t\), and consider the linear system of conics with base locus \(Y_t\). For \(t\neq 0\) this defines the quadratic transformation with base locus \(P_1,P_2,Q_t\), but on the special fiber $\PP \HH^0(\mathbb P^2,\mathcal I_{Y_0}(2))^{\vee}=L+|\mathcal O_{\mathbb P^2}(1)|$. For the family \(\mathbb P^2\times \mathbb A^1_t\to\mathbb A^1_t\) this gives an infinite order element \(\phi\) in the birational automorphism group of the generic fiber whose specialization is the identity. Explicitly, one can choose coordinates so that \(\phi\colon [x:y:z]\mapsto [x(x-ty): (x-tz)y: (x-ty)z]\).
\end{example}

It is well known that birational automorphisms on K3 surfaces extend to regular automorphisms, so in the next example the specialization homomorphism is defined on $\Bir = \Aut$.

\begin{example}\label{exmp:K3-example}
Let \(X\) be a complex K3 surface of Picard rank \(2\) obtained as the intersection of two divisors of type \((1,1)\) and \((2,2)\) in \(\PP^2\times\PP^2\). There are two projections $p_{j} \colon X \rightarrow \PP^{2}$ ($j = 1, 2$), which induce involutions $\iota_{j}$ on $X$. By \cite[Thm.~2.9]{Wehler88}, for a general such $X$ it is known that the automorphism group of $X$ is the free product $\Aut(X) \cong \ZZ/2\ZZ \ast \ZZ/2\ZZ$ generated by the involutions. On the other hand, there are special examples of such K3 surfaces where the involutions commute. In the coordinates $([X_{0}: X_{1}: X_{2}], [Y_{0}: Y_{1}: Y_{2}]) \in \PP^{2} \times \PP^{2}$, one may take the complete intersection $X_{0}$ given by the equations
\[ \sum_{i,j\in\{0,1\}} a_{ij} X_{i}Y_{j} = 0 \quad \text{and} \quad \sum_{i,j\in\{0,1,2\}} b_{ij} X_{i}^{2} Y_{j}^{2} = 0, \]
which is smooth for general coefficients $a_{ij}$ and $b_{ij}$. On \(X_0\) the covering involutions extend to (regular) involutions:
\[ \iota_{1,0} \colon ([X_{0}: X_{1}: X_{2}] , [Y_{0}: Y_{1}: Y_{2}]) \mapsto ([X_{0}: X_{1}: (-1) \cdot X_{2}] , [Y_{0}: Y_{1}: Y_{2}]), \]
\[ \iota_{2,0} \colon ([X_{0}: X_{1}: X_{2}] , [Y_{0}: Y_{1}: Y_{2}]) \mapsto ([X_{0}: X_{1}: X_{2}] , [Y_{0}: Y_{1}: (-1) \cdot Y_{2}]). \]
By construction, these involutions on $X_{0}$ automatically commute. This shows that the birational automorphism $\iota_{1}\iota_{2}\iota_{1}\iota_{2}$ on the general K3 surface $X$ has infinite order but specializes to the identity on the special fiber $X_{0}$. On the special fiber, each projection \(p_j\) contracts the conic over \([0:0:1]\), so the Picard rank jumps and the covering involution does not extend to a regular involution on the family (c.f. \cite[Thm.~2.1]{LieblichMaulik18}).

One may exhibit similar behavior on K3 surfaces of type (2,2,2) in $(\PP^{1})^{3}$, see \cite[\S 3]{Wang-thesis} and \cite[Prop.~3.5]{Schaffler2018}. For an example with Enriques surfaces, see \cite{BarthPeters-Enriques-surfaces} (c.f. \cite[IV Ex.~1.17.4]{Koll'ar-rational-curves}).

\end{example}

\begin{example}
In mixed characteristic \((0,p)\), the kernel of \(\sp_\eta\) can contain \(p\)-torsion. It is not clear if this can be accounted for by considering an additional scheme structure on $\Bir(X)$. For instance, the group of \(p\)-torsion geometric points of an elliptic curve is isomorphic to \(\ZZ/p\ZZ\times\ZZ/p\ZZ\) in characteristic \(0\), but is isomorphic to \(\ZZ/p\ZZ\) or is trivial in characteristic \(p\), so translating by a $p$-torsion point that specializes to the identity gives such an example. Similarly one can construct examples by considering $\mu_p$ actions on a scheme in mixed characteristic $(0,p)$. This happens when considering $\mu_p$-covers of schemes, and it will be an important tool in the next section.
\end{example}

\section{Applications to birational automorphisms of Fano hypersurfaces}

We now give the proofs of Theorem~\ref{thm:torsion-p^e} and Corollary~\ref{cor:no-torsion}. The key ingredients used are the specialization homomorphism for \(\Bir\), a result of the first and third authors showing that certain \(p\)-cyclic covers in characteristic \(p\) have no birational automorphisms \cite[Cor. C]{ChenStapleton-finite-BirX}, and a construction of Mori \cite{Mori75} (see also \cite[V.5.14.4]{Koll'ar-rational-curves}) that allows us to degenerate from a hypersurface to a \(p\)-cyclic cover. We begin by recalling Mori's construction:

\begin{construction}\label{construction:mori}
Let \(f,g\in R[x_0,\ldots,x_{n+1}]\) be homogeneous polynomials of degree \(pe\) and \(e\), respectively. Assume \(g^p-f\) is not uniformly \(0\). Let \(Z=(y^p-f=g-\pi y=0)\subset\mathbb P_R(1^{n+2};e)\). Then \(Z_\eta\) is isomorphic to the degree \(pe\) hypersurface \((g^p-\pi^pf=0)\subset\mathbb P^{n+1}_K\), and \(Z_0\) is isomorphic to a \(p\)-cyclic cover of the degree \(e\) hypersurface \((g=0)\subset\mathbb P^{n+1}_k\).
\end{construction}

There are two different degenerations that are most useful in our case:
\begin{itemize}
    \item A $p$-cyclic cover in mixed characteristic $(0,p)$, and
    \item Mori's construction in equicharacteristic \(0\).
\end{itemize}
By \cite[Thm.~C \& Ex.~1.7]{ChenStapleton-rational-endomorphisms}, these families have sustained separably uniruled modifications and sustained ruled modifications, respectively. Therefore we may apply Proposition~\ref{prop:specialization-hom}\eqref{item:specialization-sustained}.

\begin{proposition}\label{prop:BirtorsX-p}
Let \(p\) be a prime and let \(n,e\geq 3\) be integers such that \((p-1)e\leq n-e\leq pe-3\). Furthermore, assume \(n\) is even if \(p=2\). If \(X\subset\mathbb P^{n+1}_{\mathbb C}\) is a very general hypersurface of degree \(pe\), then any finite order element of \(\Bir(X)\)  has order a power of $p$.
\end{proposition}

\begin{proof}
The inequalities in the statement of the proposition imply that over \(\overline{\mathbb F}_p\), a general \(p\)-cyclic cover of a degree \(e\) hypersurface in \(\mathbb P^{n+1}\) has trivial birational automorphism group by \cite[Cor.~C]{ChenStapleton-finite-BirX} and is not separably uniruled by \cite[Lem.~7]{Koll'ar-hypersurfaces}. So it follows from \cite[Thm.~C]{ChenStapleton-rational-endomorphisms}, Proposition~\ref{prop:specialization-hom}\eqref{item:specialization-sustained}, and Proposition~\ref{prop:specialization-order-l}\eqref{item:mu_l-nontrivial-on-special-fiber} that for a very general such \(p\)-cyclic cover \(Z\) over $\CC$, $\Bir_{\mathbb C}(Z)$ only contains elements whose orders are \(p\)-powers. By Construction~\ref{construction:mori}, there is a family of degree \(pe\) hypersurfaces over a complex curve that degenerates to a general such \(p\)-cyclic cover. Since $Z$ is not ruled \cite[Prop.~5.12]{Koll'ar-rational-curves} and the total space has sustained ruled modifications \cite[Ex.~1.7]{ChenStapleton-rational-endomorphisms}, we may apply Proposition~\ref{prop:specialization-hom}\eqref{item:specialization-sustained}. Together with Proposition~\ref{prop:specialization-order-l}\eqref{item:mu_l-nontrivial-on-special-fiber} and the isomorphism between the geometric generic and very general fibers of the family \cite[Lem. 2.1]{Vial13}, this gives the result for a very general degree \(pe\) hypersurface over $\CC$.
\end{proof}

\begin{proof}[Proof of Theorem~\ref{thm:torsion-p^e}]
Let \(e\coloneqq \lceil \frac{n+3}{p+1}\rceil\). We will first show the result for \(d=pe\). By the comment after Theorem~\ref{thm:torsion-p^e}, we may assume that $d \leq n$ (note that this implies \(n\geq 3p\)). The assumptions in the theorem then imply that \((p-1)e\leq n-e\leq pe-3\), so by Proposition~\ref{prop:BirtorsX-p} any torsion element in the birational automorphism group of a very general hypersurface of degree \(pe\) in \(\mathbb P^{n+1}_{\mathbb C}\) has order a power of $p$.

For \(d>pe\) we prove the result by induction, showing that the degree $d-1$ result implies the degree $d$ result. To start, consider a pencil of hypersurfaces spanned by a smooth degree $d$ hypersurface and a degree $d-1$ hypersurface union with a hyperplane. Assume that the union of all three is an snc divisor. Then the total space of the pencil is singular (as the dimension of the hypersurfaces is $\ge 3$) and admits a small resolution by blowing up the hyperplane in the central fiber. After this blowup, the localization of the family at the reducible fiber has reduced snc central fiber with two components birational to the original ones. Thus the localized family has sustained ruled modifications by \cite[Ex.~1.7]{ChenStapleton-rational-endomorphisms}.

By induction the only finite order birational automorphisms of a very general degree $d-1$ hypersurface have order a power of $p$. Moreover, it is not ruled by \cite[Thm.~2]{Koll'ar-hypersurfaces}, so we may apply Proposition~\ref{prop:specialization-hom}\eqref{item:specialization-sustained} to the above degeneration to prove the result in degree $d$.
\end{proof}





\begin{proof}[Proof of Corollary~\ref{cor:no-torsion}]
Combine the results for the primes \(p=2,3\) in Theorem~\ref{thm:torsion-p^e} if \(n\) is even, and consider the primes \(p=3,5\) if \(n\) is odd.
\end{proof}

\bibliographystyle{siam}
\bibliography{Biblio.bib}

\begin{thebibliography}{10}

\bibitem{BarthPeters-Enriques-surfaces}
{\sc W.~Barth and C.~Peters}, {\em Automorphisms of {E}nriques surfaces},
  Invent. Math., 73 (1983), pp.~383--411.

\bibitem{Cantat14}
{\sc S.~Cantat}, {\em Morphisms between {C}remona groups, and characterization
  of rational varieties}, Compos. Math., 150 (2014), pp.~1107--1124.

\bibitem{Cheltsov04}
{\sc I.~A. Cheltsov}, {\em Regularization of birational automorphisms}, Mat.
  Zametki, 76 (2004), pp.~286--299.

\bibitem{ChenStapleton-finite-BirX}
{\sc N.~{Chen} and D.~{Stapleton}}, {\em {Higher index Fano varieties with
  finitely many birational automorphisms}}, arXiv e-prints,  (2021),
  p.~arXiv:2110.09568.

\bibitem{ChenStapleton-rational-endomorphisms}
\leavevmode\vrule height 2pt depth -1.6pt width 23pt, {\em {Rational
  endomorphisms of Fano hypersurfaces}}, arXiv,  (2021), p.~arXiv:2103.12207.

\bibitem{Moonen-AV}
{\sc B.~Edixhoven, G.~van~der Geer, and B.~Moonen}, {\em Abelian varieties}.
\newblock \url{https://www.math.ru.nl/~bmoonen/BookAV/Quotients.pdf}.

\bibitem{Hanamura87}
{\sc M.~Hanamura}, {\em On the birational automorphism groups of algebraic
  varieties}, Compositio Mathematica, 63 (1987), pp.~123--142.

\bibitem{Hanamura88}
\leavevmode\vrule height 2pt depth -1.6pt width 23pt, {\em Structure of
  birational automorphism groups, i: non-uniruled varieties}, Inventiones
  mathematicae, 93 (1988), pp.~383--403.

\bibitem{Koll'ar-hypersurfaces}
{\sc J.~Koll\'{a}r}, {\em Nonrational hypersurfaces}, J. Amer. Math. Soc., 8
  (1995), pp.~241--249.

\bibitem{Koll'ar-rational-curves}
\leavevmode\vrule height 2pt depth -1.6pt width 23pt, {\em Rational curves on
  algebraic varieties}, vol.~32 of Ergebnisse der Mathematik und ihrer
  Grenzgebiete. 3. Folge. A Series of Modern Surveys in Mathematics [Results in
  Mathematics and Related Areas. 3rd Series. A Series of Modern Surveys in
  Mathematics], Springer-Verlag, Berlin, 1996.

\bibitem{Kollar19}
{\sc J.~Koll{\'a}r}, {\em The rigidity theorem of
  {F}ano--{S}egre--{I}skovskikh--{M}anin--{P}ukhlikov--{C}orti--{C}heltsov--de
  {F}ernex--{E}in--{M}usta{\c{t}}{\u{a}}--{Z}huang}, in Birational Geometry of
  Hypersurfaces, Springer, 2019, pp.~129--164.

\bibitem{LieblichMaulik18}
{\sc M.~Lieblich and D.~Maulik}, {\em A note on the cone conjecture for {K}3
  surfaces in positive characteristic}, Math. Res. Lett., 25 (2018),
  pp.~1879--1891.

\bibitem{LS22}
{\sc H.-Y. Lin and E.~Shinder}, {\em Motivic invariants of birational maps},
  arXiv preprint arXiv:2207.07389,  (2022).

\bibitem{Liu-AG-book}
{\sc Q.~Liu}, {\em Algebraic geometry and arithmetic curves}, vol.~6 of Oxford
  Graduate Texts in Mathematics, Oxford University Press, Oxford, 2002.
\newblock Translated from the French by Reinie Ern\'{e}, Oxford Science
  Publications.

\bibitem{Matsumura63}
{\sc H.~Matsumura}, {\em On algebraic groups of birational transformations},
  Atti Accad. Naz. Lincei Rend. Cl. Sci. Fis. Mat. Natur.(8), 34 (1963),
  pp.~2--4.

\bibitem{MM63}
{\sc H.~Matsumura and P.~Monsky}, {\em On the automorphisms of hypersurfaces},
  Journal of Mathematics of Kyoto University, 3 (1963), pp.~347--361.

\bibitem{MatsusakaMumford64}
{\sc T.~Matsusaka and D.~Mumford}, {\em Two fundamental theorems on
  deformations of polarized varieties}, Amer. J. Math., 86 (1964),
  pp.~668--684.

\bibitem{Mori75}
{\sc S.~Mori}, {\em On a generalization of complete intersections}, J. Math.
  Kyoto Univ., 15 (1975), pp.~619--646.

\bibitem{Persson77}
{\sc U.~Persson}, {\em On degenerations of algebraic surfaces}, Mem. Amer.
  Math. Soc., 11 (1977), pp.~xv+144.

\bibitem{Pukhlikov-index2}
{\sc A.~V. Pukhlikov}, {\em Birational geometry of {F}ano hypersurfaces of
  index two}, Math. Ann., 366 (2016), pp.~721--782.

\bibitem{Schaffler2018}
{\sc L.~Schaffler}, {\em K3 surfaces with $\mathbb{Z}_{2}^{2}$ symplectic
  action}, Rocky Mountain Journal of Mathematics, 48 (2018), pp.~2347--2383.

\bibitem{stacks-project}
{\sc T.~{Stacks Project Authors}}, {\em \textit{Stacks Project}}.
\newblock \url{https://stacks.math.columbia.edu}, 2018.

\bibitem{Vial13}
{\sc C.~Vial}, {\em Algebraic cycles and fibrations}, Doc. Math., 18 (2013),
  pp.~1521--1553.

\bibitem{Wang-thesis}
{\sc L.~Wang}, {\em Rational points and canonical heights on varieties with
  many elliptic fibrations}, ProQuest LLC, Ann Arbor, MI, 1994.
\newblock Thesis (Ph.D.)--Harvard University.

\bibitem{Wehler88}
{\sc J.~Wehler}, {\em {$K3$}-surfaces with {P}icard number {$2$}}, Arch. Math.
  (Basel), 50 (1988), pp.~73--82.

\end{thebibliography}

\footnotesize{
\textsc{Department of Mathematics, Harvard University, Cambridge, Massachusetts 02138} \\
\indent \textit{E-mail address:} \href{mailto:nathanchen@math.harvard.edu}{nathanchen@math.harvard.edu}

\textsc{Department of Mathematics, University of Michigan, Ann Arbor, Michigan 48109} \\
\indent \textit{E-mail address:} \href{mailto:lenaji.math@gmail.com}{lenaji.math@gmail.com}

\textsc{Department of Mathematics, University of Michigan, Ann Arbor, Michigan 48109} \\
\indent \textit{E-mail address:} \href{mailto:dajost@umich.edu}{dajost@umich.edu}
}

\end{document}